\documentclass[12pt,english,reqno]{amsart}
\usepackage{geometry}
\usepackage{amsmath,amssymb,amsthm,bbm,mathtools,comment}
\usepackage{amsfonts}
\usepackage[shortlabels]{enumitem}
\usepackage[pdftex,colorlinks,backref=page,citecolor=blue]{hyperref}
\usepackage[mathscr]{euscript}
\usepackage[usenames,dvipsnames]{color}
\usepackage{tikz,babel,adjustbox}
\usepackage[numbers]{natbib}
\usepackage{graphicx}
\usepackage{caption}
\usepackage{subcaption}
\usepackage{verbatim}
\usepackage{array}
\usepackage[frame,cmtip,arrow,matrix,line,graph,curve]{xy}
\usepackage{graphpap, pstricks}
\usepackage{etoolbox}
\usepackage{pifont}
\usepackage[final]{microtype}
\usepackage{cmtiup}
\usepackage[noabbrev,capitalize]{cleveref}

\hypersetup{pdfpagemode=UseNone,pdfstartview={XYZ null null 1.00}}
\usetikzlibrary{shapes.misc,calc,intersections,patterns,decorations.pathreplacing}
\usetikzlibrary{arrows,arrows.meta,shapes,positioning,decorations.markings}
\tikzstyle arrowstyle=[scale=1]

\allowdisplaybreaks

\makeatletter
\def\@settitle{\begin{center}%
        \bfseries\Large
        \@title
    \end{center}%
}
\patchcmd{\@setauthors}{\MakeUppercase}{\normalsize}{}{}
\makeatother

\theoremstyle{plain}
\newtheorem{theorem}{Theorem}[section]
\newtheorem{lemma}[theorem]{Lemma}

\theoremstyle{remark}

\newcommand{\beq}[1]{\begin{equation}\label{#1}}
\newcommand{\enq}[0]{\end{equation}}

\def\E{\mathbb{E}}

\def\N{\mathbb{N}}
\def\RR{\mathbb{R}_{\ge 0}}
\DeclareMathOperator\fix{fix}
\DeclareMathOperator\per{per}
\DeclareMathOperator\intt{int}

\makeatletter
\def\imod#1{\allowbreak\mkern10mu({\operator@font mod}\,\,#1)}
\makeatother

\begin{document}

\title{Elementary symmetric polynomials under the fixed point measure}

\author{Ayush Khaitan}
\address{Department of Mathematics, Rutgers University, Piscataway, NJ 08854, USA}
\email{ayush.khaitan@rutgers.edu}

\author{Ishan Mata}
\address{Theoretical Statistics and Mathematics Unit, Indian Statistical Institute, New Delhi 110016, India}
\email{ishanmata@gmail.com}

\author{Bhargav Narayanan}
\address{Department of Mathematics, Rutgers University, Piscataway, NJ 08854, USA}
\email{narayanan@math.rutgers.edu}

\date{8 April, 2025}

\begin{abstract}
	We identify a surprising inequality satisfied by elementary symmetric polynomials under the action of the fixed point measure of a random permutation. Concretely, for any collection of $n$ non-negative real numbers $a_1, \dots, a_n \in \RR$, we prove that
	\[
		\frac{1}{n!} \sum_{\pi \in S_n} \left[\prod_{\{i:i=\pi(i)\}} a_i\right] \ge \frac{1}{\binom{n}{2}} \sum_{S \in\binom{[n]}{2}} \left[ \left(\prod_{\{i \in S\}} a_i \right)^{1/2}\right],
	\]
	and this bound is sharp. To prove this elementary inequality, we construct a collection of differential operators to set up a monotone flow that then allows us to establish the inequality.
\end{abstract}
\maketitle

\section{Introduction}
For integers $0 \le k \le n$, the elementary symmetric polynomial $e_k(\mathbf{x})$ in the variables $\mathbf{x} = (x_1, \dots, x_n)$ is defined by
\[
	e_k (\mathbf{x}) = \sum_{S \in \binom{[n]}{k} } \prod_{i \in S} x_i,
\]
and its normalised counterpart $s_k(\mathbf{x})$ is given by $s_k(\mathbf{x}) = e_k(\mathbf{x}) / \binom{n}{k}$; here, we write $\binom{[n]}{k}$ for the family of $k$-element subsets of $[n] = \{1, \dots, n\}$, and as is standard, interpret any product over the empty set as being one, so that $e_0(\mathbf{x}) = s_0(\mathbf{x}) = 1$.

Inequalities involving the elementary symmetric polynomials are a classical subject of study. To give but two famous examples, for any sequence $\mathbf{a} = (a_1, \dots, a_n)$ of non-negative reals, a well-known theorem of Newton asserts that the sequence $\{ s_k(\mathbf{a}) \}_{k=0}^n $ is log-concave, and this in turn implies another well-known result of Maclaurin that the sequence $\{( s_k(\mathbf{a}) )^{1/k}\}_{k=1}^n $ is non-increasing; for a survey of these and other related results, see~\citep{HLP, CSM}.

In this paper, we shall prove a new inequality involving the elementary symmetric polynomials, one that arises from the action of the fixed point measure of a random permutation on these polynomials. Writing $\fix(\pi)$ for the set of fixed points of a permutation $\pi:[n] \to [n]$ and $S_n$ for the symmetric group over $[n]$, our main result is as follows.

\begin{theorem}\label{thm:main}
	For any collection $\mathbf{a} = (a_1, \dots, a_n)$ of $n \ge 2$ non-negative reals, we have
	\beq{mainbd}
	\E_{\pi \in S_n} \left[\prod_{i\in\fix(\pi)} a_i\right] \ge \E_{S \in\binom{[n]}{2}} \left[ \left(\prod_{i \in S} a_i \right)^{1/2}\right],
	\enq
	where both expectations are over the respective uniform measures; furthermore, provided $n \ge 3$, equality holds if and only if $a_i = 1$ for all $1 \le i \le n$.
\end{theorem}
In terms of the rencontres numbers $D(n,k)$ that count the number of permutations of an $n$-element set with exactly $k$ fixed points, Theorem~\ref{thm:main} may be reformulated as follows: for any collection $\mathbf{a} = (a_1, \dots, a_n)$ of $n$ non-negative reals, we have
\[
	\sum_{k =0}^n d({n,k}) s_k\left(\mathbf{a}\right) \ge s_2\left(\sqrt{\mathbf{a}}\right),
\]
where $d(n,k)$ is the fraction $D(n,k) / n!$ and $\sqrt{\mathbf{a}} = (\sqrt{a_1}, \dots, \sqrt{a_n})$ is the vector of non-negative square roots of the elements of $\mathbf{a}$.

While the statement of Theorem~\ref{thm:main} appears to be somewhat unmotivated and magical on the surface, it arises naturally in the context of lower bounds for the \emph{matrix permanent}, a subject about which we hope to say more elsewhere. Here, we content ourselves with pointing out that the left-hand side of the inequality in Theorem~\ref{thm:main} is easily seen to be $\per(M(\mathbf{a})) / n!$, where $M(\mathbf{a})$ is the $n\times n$ matrix whose diagonal entries are $a_1, \dots, a_n$ and whose off-diagonal entries are all one.

Our proof of Theorem~\ref{thm:main} relies on a somewhat unusual approach. Concretely, we construct a collection of differential operators to set up a monotone flow that allows us to locate (within some compact subset of the domain where it suffices to establish the inequality) the critical points of the difference between the left-hand side and the right-hand side of~\eqref{mainbd}. To get a sense of the difficulties involved in proving Theorem~\ref{thm:main}, and to appreciate why our proof of Theorem~\ref{thm:main} requires such machinery, it will be helpful to consider the result in some special cases and for small values of $n \in \N$.

When all the $a_i$ are equal, say to $a$, then the right-hand side of~\eqref{mainbd} is just $a$, while the left-hand side of~\eqref{mainbd} is $\E[a^X]$, where $X$ is the number of fixed points of a uniformly random permutation; since $\E[X]=1$, the result in this case follows from convexity via Jensen's inequality. It is tempting to look for a similar argument that relies on simple convexity in the general case; such a proof might well exist, but we have tried hard to find such an argument and have nothing to show for our efforts.

The statement of Theorem~\ref{thm:main} is easy to check for $n=2$ since it asserts (after scaling by a factor of $2!=2$) that $1+a_1 a_2 \ge 2\sqrt{a_1 a_2}$ for any $a_1, a_2 \ge 0$; this is of course immediate from the AM--GM inequality. Our main result is already non-trivial for $n=3$. In this case, the inequality asserts (now, after scaling by a factor of $3!=6$) that for any $a_1, a_2, a_3 \ge 0$, we have
\beq{n3}
2 + a_1 + a_2 + a_3 + a_1a_2a_3 \ge 2\sqrt{a_1a_2} + 2\sqrt{a_2a_3} + 2 \sqrt{a_1a_3}.
\enq
It is not too difficult to prove~\eqref{n3}, but what is striking to us is the fact that every elementary proof of~\eqref{n3} that we know of appears to \emph{break symmetry} in the variables $(a_1, a_2, a_3)$ in some way (even though~\eqref{n3} is itself completely symmetric). For example, we may deduce~\eqref{n3} from a well-known inequality of Schur~\citep{CSM} that says that for any $x,y,z \ge 0$, we have $x(x-y)(x-z) + y(y-z)(y-x) + z(z-x)(z-y) \ge 0$; that said, the proof of this inequality itself involves some symmetry breaking, and indeed, it is known, see~\citep{Mitrin,Wu}, that on account of this symmetry breaking, there are no natural extensions of Schur’s inequality that apply to four or more non-negative reals.

In fact, we can quantify the difficulty in proving~\eqref{mainbd} a little more precisely. First, it is not too hard to show that the validity of~\eqref{n3} (and thus~\eqref{mainbd} as well) has no `sum of squares'-certificate in the ring of polynomials. Second, there is a rather powerful inequality due to Muirhead~\citep{Muirhead} that allows one to compare two symmetric means of non-negative reals. The right-hand side of~\eqref{mainbd} is a symmetric Muirhead mean of the numbers $a_1, \dots, a_n$ (associated with the vector $(1/2,1/2,0, \dots, 0)$), while the left-hand side is a weighted linear combination of the elementary symmetric polynomials, which are themselves Muirhead means of the numbers $a_1, \dots, a_n$ as well. However, it can be shown that even~\eqref{n3}, the specialisation of~\eqref{mainbd} to three variables, is \emph{not in the Muirhead semiring} in the sense of~\citep{Cuttler}, or in other words, that~\eqref{n3} does not follow (in a logically precise sense) from Muirhead's inequality.

The rest of this paper is organised as follows. Our proof of Theorem~\ref{thm:main} via the construction of a monotone flow is given in Section~\ref{sec:proof}. We conclude with a discussion of open questions in Section~\ref{sec:conc}.

\section{Proof of the main result}\label{sec:proof}
We shall prove Theorem~\ref{thm:main}, for $n \ge 4$, through an analysis of the critical points of an appropriately defined real-valued function over $\RR^n$. To carry out this analysis, we need some notation.

Let $c_k$ denote the number of derangements of $[k]$, i.e., the number of permutations of $[k]$ with no fixed points. For $n\in\N$, we define the functions $L_n(\mathbf{x})$ and $R_n(\mathbf{x})$ in the real variables $\mathbf{x} = (x_1, \dots, x_n)$ to be the left-hand side and right-hand side of~\eqref{mainbd}, respectively; more precisely, we have
\[
	L_n(\mathbf{x}) = \frac{1}{n!} \sum_{i=0}^n c_{n-i} e_i(\mathbf{x}) = \frac{1}{n!} \sum_{\pi \in S_n} \prod_{i \in \fix(\pi)} x_i,
\]
and
\[
	R_n(\mathbf{x}) = s_2\left(\sqrt{\mathbf{x}}\right) = \frac{1}{\binom{n}{2}} \sum_{S \in \binom{[n]}{2}} \left(\prod_{i \in S} x_i\right)^{1/2},
\]
where $\sqrt{\mathbf{x}} = (\sqrt{x_1}, \sqrt{x_2}, \ldots, \sqrt{x_n})$ is the vector of non-negative square roots of the entries of $\mathbf{x}$.

Our goal then is to prove that the function
\[f_n(\mathbf{x}) = L_n(\mathbf{x}) - R_n(\mathbf{x})\]
is non-negative for all $\mathbf{x} \in \RR^n$. In broad strokes, our proof of Theorem~\ref{thm:main} proceeds as follows.
\begin{enumerate}
	\item First, we establish that it suffices to show that $f_n \ge 0$ over a compact subset $\mathfrak{C} \subset \RR^n$, and we also show that $f_n$ does not attain its minimum on the boundary of $\mathfrak{C}$.
	\item We know that a differentiable function on a compact set attains its minimum either at a point on the boundary or at a critical point in the interior, so we finish by showing that $f_n$ has a unique critical point $(1,\dots,1)$ in the interior of $\mathfrak{C}$, and that $f_n$ is equal to zero at this point.
\end{enumerate}

The first step is fairly straightforward, but the second step is somewhat delicate. To show that $(1, \dots, 1)$ is the unique critical point of $f_n$ in $\mathfrak{C}$, we need to show that the gradient $\nabla f_n$ is non-vanishing at every other point in the interior of $\mathfrak{C}$. However, this appears to be quite difficult; indeed, it is unclear, a priori, which components of $\nabla f_n$ ought to be non-zero at any given point. We circumvent this difficulty by instead constructing a monotone flow; concretely, we construct a family of differential operators whose action on $f_n$ allows us to exhibit, at each $\mathbf{x} \ne (1, \dots, 1)$, a direction $d(\mathbf{x})$ in which $f_n$ is decreasing.

We now turn to the details of how to execute this strategy.

\begin{proof}[Proof of Theorem~\ref{thm:main}]
	We shall first prove the inequality assuming $n\geq 4$; the cases where $n\le 3$ are handled separately at the end with more ad hoc arguments.

	Before we proceed, we remind the reader that the number $c_k$ of derangements of $[k]$ is given by
	\[c_k=k!\sum\limits_{j=0}^k\frac{(-1)^j}{j!};\]
	in particular, $k!/3 \le c_k \le k!/2 $ for $k\ge 2$.

	We first show that it suffices to restrict our attention to a compact subset of $\RR^n$.

	\begin{lemma}\label{lem:compact}
		For any $n \ge 4$, the function $f_n(\mathbf{x})$ is non-negative for all $\mathbf{x} \in \RR^n$ if and only if it is non-negative for all $\mathbf{x} \in  \mathfrak{C}$, where
		\[
			\mathfrak{C} = \{ \mathbf{x} = (x_1, \dots, x_n): x_1 \ge 0,\dots,x_n \ge 0 \text{ and } \sum_i x_i \le 6n \}.
		\]
	\end{lemma}
	\begin{proof}
		We claim that for all $n \ge 4$, $f_n(\mathbf{x}) \ge 5/6$ whenever $\sum_i x_i \ge 6n$. To see this, note that
		\begin{align*}
			f_n (\mathbf{x}) & = \frac{1}{n!} \sum_{k=0}^n c_{n-k}e_k\left(\mathbf{x}\right) - \frac{1}{\binom{n}{2}} e_2\left(\sqrt{\mathbf{x}}\right)                                      \\
			                 & \geq \frac{1}{n!} \left( c_n + c_{n-1} e_1(\mathbf{x}) + c_{n-2} e_2(\mathbf{x}) \right) - \frac{1}{\binom{n}{2}} e_2\left(\sqrt{\mathbf{x}}\right)           \\
			                 & \ge \frac{1}{3 n!} \left( n! + (n-1)! \cdot e_1(\mathbf{x}) + (n-2)! \cdot e_2(\mathbf{x}) \right) - \frac{1}{\binom{n}{2}} e_2\left(\sqrt{\mathbf{x}}\right) \\
			                 & \ge \frac{1}{3} + \frac{e_1(\mathbf{x})}{3n} + \frac{1}{\binom{n}{2}} \sum_{i<j}\left( x_i x_j / 6 - \sqrt{x_ix_j} \right),
		\end{align*}
		where the bound $c_{n-2} \ge (n-2)!/3$ relies on our assumption that $n \ge 4$. Now, it is easily verified that $y/6 - \sqrt{y} \ge -3/2$ for all $y \ge 0$, so it follows from the assumption that $e_1(\mathbf{x}) = \sum_i x_i \ge 6n$ that
		\begin{align*}
			f_n (\mathbf{x}) & \ge \frac{1}{3}+ \frac{6n}{3n} + \frac{1}{\binom{n}{2}} \sum_{i<j}\left( -3/2 \right) \\
			                 & = 1/3 + 2 -3/2 = 5/6,
		\end{align*}
		as claimed.
	\end{proof}

	Hence, it suffices to show that $f_n(\mathbf{x}) \ge 0$ for $\mathbf{x} \in \mathfrak{C}$. Next, we claim that $f_n(\mathbf{x})$ does not attain its minimum on the boundary of $\mathfrak{C}$. The boundary of $\mathfrak{C}$ is $\{\mathbf{x} \in S: e_1(\mathbf{x})=6n \vee x_1=0 \vee \dots \vee x_n=0\}$. Lemma~\ref{lem:compact} tells us that $f_n(\mathbf{x}) \ge 5/6$ when $e_1(\mathbf{x})=6n$, while it is easy to see that $\mathbf{1} = (1, \dots, 1)$ is in $\mathfrak{C}$ and that $f_n(\mathbf{1}) = 0$. So to show that $f_n(\mathbf{x})$ does not attain its minimum on the boundary of $\mathfrak{C}$, it suffices to consider those points $\mathbf{x}$ in the boundary of $\mathfrak{C}$ where $x_i=0$ for some $i \in [n]$. For any $i \in [n]$, as $x_i\to 0^+$ with the other $x_j$ fixed with at least one $x_j \ne 0$, we have $\partial f_n /\partial x_i = -\Theta(1/\sqrt{x_i}) \to -\infty$. Hence, at any point on the boundary of $\mathfrak{C}$ other than $\mathbf{0} = (0, \dots, 0)$ where some $x_i = 0$, $f_n$ is strictly decreasing in the direction of  $\mathbf{b}_i$, $i$-th standard basis vector, and this vector is clearly inward pointing. On the other hand, at $\mathbf{0}$, $f_n$ is easily seen to be strictly decreasing in the (inward pointing) direction of $\mathbf{1} = (1,\dots,1)$; indeed, we see that $f_n(t\mathbf{1}) = (c_n/n!) - (1 - c_{n-1}/(n-1)!)t + \Theta(t^2)$ as $t\to0^+$, and this shows (since $c_{n-1} < (n-1)!$) that $f_n(t\mathbf{1})$ is strictly decreasing in $t$ for small $t \ge 0$. Hence, we conclude that $f_n$ does not attain its minimum on the boundary of $\mathfrak{C}$.

	On a compact set, a continuous function may only attain its minimum at points on the boundary or at critical points in the interior. We shall finish the proof by showing that $f_n$ has exactly one critical point, namely the point $\mathbf{1}$, in the interior of $\mathfrak{C}$. Since $f_n(\mathbf{1}) = 0$, we may then conclude that $f_n(\mathbf{x})\ge 0$ for all $\mathbf{x} \in \RR^n$, and that equality is attained only at the point $\mathbf{1}$.

	We begin with some preliminary calculations. The partial derivatives of $L_n$ are given by
	\[
		\frac{\partial L_n}{\partial x_i} = \frac{\partial}{\partial x_i} \left( \frac{1}{n!} \sum_{\pi \in S_n} \prod_{j \in \fix(\pi)} x_j \right) = \frac{(n-1)!}{n!} \left( \frac{1}{(n-1)!} \sum_{\sigma \in S_{n-1}^{(i)}} \prod_{j \in \fix(\sigma)} x_j \right) = \frac{1}{n} L_{n-1}(\widehat{\mathbf{x}}_i)
	\]
	where $S_{n-1}^{(i)}$ denotes the symmetric group acting on $[n]\setminus\{i\}$, and $\widehat{\mathbf{x}}_i$ is the vector $\mathbf{x}$ with the $i$-th coordinate removed. Next, the partial derivatives of $R_n$ are given by
	\[
		\frac{\partial R_n}{\partial x_i} = \frac{\partial}{\partial x_i} \left( \frac{1}{\binom{n}{2}} \sum_{1 \le j < k \le n} \sqrt{x_j x_k} \right)
		= \frac{1}{n(n-1)} \sum_{j \ne i} \sqrt{\frac{x_j}{x_i}}.
	\]
	Therefore, the partial derivatives of $f_n(\mathbf{x}) = L_n(\mathbf{x}) - R_n(\mathbf{x})$ are given by
	\begin{equation} \label{eq:partial_f}
		\frac{\partial f_n}{\partial x_i} = \frac{1}{n} L_{n-1}(\widehat{\mathbf{x}}_i) - \frac{1}{n(n-1)} \sum_{j \ne i} \sqrt{\frac{x_j}{x_i}}.
	\end{equation}

	With these formulae in hand, we first verify that $\mathbf{1}$ is indeed a critical point of $f_n$.
	\begin{lemma}
		The point $\mathbf{1}$ is a critical point of $f_n$, and $f_n(\mathbf{1})=0$.
	\end{lemma}
	\begin{proof}
		First, note that
		\[
			f_n(\mathbf{1}) = L_n(\mathbf{1}) - R_n(\mathbf{1}) = \E_{\pi \in S_n}[1] - \E_{S \in \binom{[n]}{2}}[1] = 1 - 1 = 0.
		\]
		Next, using~\eqref{eq:partial_f}, we get that
		\[
			\frac{\partial f_n}{\partial x_i}\bigg|_{\mathbf{x}=\mathbf{1}} = \frac{1}{n} (1) - \frac{1}{n(n-1)} (n-1) = \frac{1}{n} - \frac{1}{n} = 0
		\]
		for all $i\in[n]$, establishing the claim.
	\end{proof}

	Next, we show that certain regions of $\RR^n$ are free of critical points of $f_n$. We first show that there are no critical points in the region $(1,\infty)^n$.
	\begin{lemma}\label{all>1}
		No critical points of the function $f_n$ lie in the region $(1,\infty)^n$.
	\end{lemma}
	\begin{proof}
		We bound the second derivatives of $f = f_n$ at each point $\mathbf{x} \in (1,\infty)^n$. First, we have
		\[ \frac{\partial^2 f}{\partial x_i^2}= \frac{1}{2 n (n-1)} \sum_{j\neq i} \sqrt{\frac{x_j}{x_i^3}} >0. \]
		Also, for $j\neq i$, we have
		\[\frac{\partial^2 f}{\partial x_j\partial x_i}= \frac{1}{n(n-1)}L_{n-2}(\widehat{\mathbf{x}}_{i,j}) -\frac{1}{2n(n-1)}\frac{1}{\sqrt{x_ix_j}},\]
		where $\widehat{\mathbf{x}}_{i,j}$ is the vector $\mathbf{x}$ with the $i$-th and $j$-th coordinates removed.
		Since $\mathbf{x} \in (1,\infty)^n$, we see that $L_{n-2}(\widehat{\mathbf{x}}_{i,j}) \ge 1$ and $\sqrt{x_i x_j} \ge 1$, so we then have
		\[\frac{\partial^2 f}{\partial x_j\partial x_i}\ge \frac{1}{n(n-1)} - \frac{1}{2n(n-1)} \ge \frac{1}{2n(n-1)} > 0.\]

		Recall that \[\frac{\partial f}{\partial x_i}\bigg|_{\mathbf{x}=\mathbf{1}}=0\] for each $i\in [n]$.
		Consequently, each derivative ${\partial f}/{\partial x_i}$ is positive everywhere in the region $(1,\infty)^n$, so there are no critical points in this region as claimed.
	\end{proof}

	Next, we show that there are no critical points in the region $(0,1)^n$.

	\begin{lemma}\label{all<1}
		No critical points of the function $f_n$ lie in the region $(0,1)^n$.
	\end{lemma}
	\begin{proof}
		Consider $\langle \nabla f (\mathbf{x}), \mathbf{1} \rangle$ for $f = f_n$. This is given by
		\[\sum\limits_{i=1}^n\frac{\partial f}{\partial x_i}=\sum\limits_{i=1}^n\left[\frac{1}{n}L_{n-1}(\widehat{\mathbf{x}}_i) -\frac{1}{n(n-1)}\sum\limits_{j\in[n]\setminus \{i\}}\sqrt{\frac{x_j}{x_i}}\right].\]
		For any $\mathbf{x} \in (0,1)^n$, it is clear that $L_{n-1}(\widehat{\mathbf{x}}_i) < 1$ for each $i \in [n]$. Hence, \[\sum\limits_{i=1}^n \frac{1}{n}L_{n-1}(\widehat{\mathbf{x}}_i)<1.\]
		On the other hand, as $\sqrt{\frac{x_j}{x_i}}+\sqrt{\frac{x_i}{x_j}} \ge 2$, we have that
		\[\frac{1}{n(n-1)}\sum\limits_{i=1}^n\sum\limits_{j\in [n]\setminus \{i\}}\sqrt{\frac{x_j}{x_i}} \ge1.\]
		Therefore, we have $\langle \nabla f (\mathbf{x}), \mathbf{1} \rangle<0$ for any $\mathbf{x} \in (0,1)^n$, so there are no critical points in this region as claimed.
	\end{proof}

	Now, fix some $\mathbf{x} \in \intt \mathfrak{C} \setminus \{(0,1)^n \cup \{\mathbf{1}\} \cup (1,\infty)^n\}$. To show that $\mathbf{x}$ is not a critical point, we construct a monotone flow, i.e., we exhibit a direction $d(\mathbf{x})$ at each such $\mathbf{x}$ such that $\langle \nabla f_n(\mathbf{x}), d(\mathbf{x}) \rangle > 0$, thereby showing that $\nabla f_n(\mathbf{x}) \ne \mathbf{0}$ and that $\mathbf{x}$ is not a critical point. Hence, suppose for a contradiction that $\mathbf{x}$ is a critical point. As $f_n$ is a symmetric function of the $x_i$'s, we may assume without loss of generality that $0 < x_1\le \dots\le x_n$. Lemmas~\ref{all>1} and~\ref{all<1} tell us that $\mathbf{x} \notin (0,1)^n \cup (1,\infty)^n$. Thus, we may also assume that $x_1<x_n$, and that $x_1\le 1 \le x_n$. To finish, we will show that there exists a direction $d(\mathbf{x})$ at $\mathbf{x}$ such that $\langle \nabla f_n (\mathbf{x}), d(\mathbf{x}) \rangle > 0$. In more detail, for every $k\in [n-1]$, we shall consider the action of the differential operators
	\[
		\mathcal{O}_{k}= x_k \frac{\partial}{\partial x_k} - x_n \frac{\partial}{\partial x_n}
	\]
	on $f_n$ and show that $\mathcal{O}_{k} f_n (\mathbf{x}) < 0$ for the largest $k \in [n-1]$ such that $x_k < x_{k+1}$ (and such a $k$ exists by our assumption that $x_1  < x_n$); in other words, we may take $d(\mathbf{x})$ to be $x_k \mathbf{b}_k - x_n \mathbf{b}_n$ for some $k \in [n-1]$, where $(\mathbf{b}_1, \dots, \mathbf{b}_n)$ denotes the standard basis.

	First, we compute $\mathcal{O}_{k} f_n(\mathbf{x})$ for $k \in [n-1]$. With
	\[
		S_{k} (\mathbf{x}) = \sum_{i=1}^{n-2} c_{n-i} e_{i-1}(\widehat{\mathbf{x}}_{k,n}) - (n-2)!\sum\limits_{i\neq k,n} \frac{\sqrt{x_i}}{\sqrt{x_{k}}+\sqrt{x_n}},
	\]
	it follows from~\eqref{eq:partial_f} after some algebraic manipulation that
	\[\mathcal{O}_{k}f_n(\mathbf{x}) =\frac{1}{n!}(x_k- x_n) S_k(\mathbf{x}).\]
	In what follows, we shall show that $S_k (\mathbf{x})>0$ for some $k\in [n-1]$.

	\begin{lemma}\label{1equal}
		If $x_{n-1}<x_n$, then $S_{n-1} (\mathbf{x})>0$.
	\end{lemma}

	\begin{proof}
		We have
		\begin{align*}
			S_{n-1}(\mathbf{x}) & = \sum_{i=1}^{n-2} c_{n-i} e_{i-1}(\widehat{\mathbf{x}}_{n-1,n}) - (n-2)!\sum_{i=1}^{n-2} \frac{\sqrt{x_i}}{\sqrt{x_{n-1}}+\sqrt{x_n}} \\
			                    & \ge c_{n-1} + c_{n-2} \sum_{i=1}^{n-2} x_i - (n-2)!\sum_{i=1}^{n-2} \frac{\sqrt{x_i}}{\sqrt{x_{n-1}}+\sqrt{x_n}}                       \\
			                    & =  \sum\limits_{i=1}^{n-2} F_{n-1}(x_i),
		\end{align*}
		where $F_{n-1}(x)$ is given by
		\beq{f-eq}
		F_{n-1}(x)=\frac{c_{n-1}}{n-2}+c_{n-2}x-(n-2)!\frac{\sqrt{x}}{\sqrt{x_{n-1}}+\sqrt{x_n}}.
		\enq
		We shall prove that $F_{n-1}(x)>0$ for all $x\in (0,x_{n-1}]$, implying that $S_{n-1}(\mathbf{x})>0$.

		Indeed, for $x \le x_{n-1}$, put $y = \sqrt{{x}/{x_n}}$ and note that
		\begin{align*}
			F_{n-1}(x) & =\frac{c_{n-1}}{n-2}+c_{n-2}x-(n-2)!\frac{\sqrt{x}}{\sqrt{x_{n-1}}+\sqrt{x_n}}            \\
			           & \geq \frac{(n-1)!}{3(n-2)}+\frac{(n-2)!}{3}x - (n-2)!\frac{\sqrt{x}}{\sqrt{x}+\sqrt{x_n}} \\
			           & =\frac{(n-1)!}{3(n-2)}+\frac{(n-2)!}{3}y^2x_n -(n-2)!\frac{y}{y+1};
		\end{align*}
		here, we use the fact that $n\ge 4$ to bound $c_{n-2} \ge (n-2)!/3$. Since $x_n \ge 1$, we then have
		\begin{align*}
			F_{n-1}(x) & \geq \frac{(n-1)!}{3(n-2)}+\frac{(n-2)!}{3}y^2 -(n-2)!\frac{y}{y+1} \\
			           & \ge (n-2)!\left[\frac{1}{3}+\frac{1}{3}y^2-\frac{y}{y+1}\right].
		\end{align*}
		The function
		\[\frac{1}{3}+\frac{1}{3}y^2-\frac{y}{y+1}\]
		is easily checked to be positive for all $y \in (0,1]$. Hence, $F_{n-1}(x)>0$ for all $x \in (0,x_{n-1}]$, implying that $S_{n-1}(\mathbf{x})>0$.
	\end{proof}

	Next, we outline how one may show that $S_{n-2}(\mathbf{x})>0$ when $x_{n-2}<x_{n-1}=x_n$. Before we do so, it will be helpful to generalise the definition of $F_{n-1}$ above, and we define $F_k(x)$ for each $k\in [n-2]$ by
	\[F_k(x) = \frac{c_{n-1}}{n-2} + c_{n-2}x - (n-2)!\frac{\sqrt{x}}{\sqrt{x_{k}}+\sqrt{x_n}}.\]

	As before, it is clear that $S_{n-2}(\mathbf{x}) \ge \sum_{i \ne n-2, n} F_{n-2}(x_i)$. While $F_{n-2}(x_i) > 0$ for each $i \in [n-3]$ as in the proof of Lemma~\ref{1equal}, the term $F_{n-2}(x_{n-1})$ can be negative when $x_{n-2} < x_{n-1} = x_n$. We get around this fact by pairing up $F_{n-2}(x_{n-1})$ with $F_{n-2}(x_1)$, say. Indeed, since $x_{n-1} = x_n$ and $x_n \ge 1$, and since $y=\sqrt{{x_1}/{x_n}}$ is at most $z=\sqrt{{x_{n-2}}/{x_n}}$, we have
	\begin{align*}
		F_{n-2}(x_1)+F_{n-2}(x_{n-1}) & \ge(n-2)!\left[\frac{1}{3}+\frac{1}{3}y^2 x_n-\frac{y}{z+1}+\frac{1}{3}+\frac{1}{3} x_{n-1}-\frac{1}{z+1}\right] \\
		                              & \ge(n-2)!\left[ \frac{1}{3}+\frac{1}{3}y^2-\frac{y}{z+1}+\frac{1}{3}+\frac{1}{3}-\frac{1}{z+1}\right]            \\
		                              & \ge(n-2)!\left[ 1+\frac{1}{3}y^2-\frac{y+1}{z+1}\right] \ge \frac{(n-2)!}{3}y^2 > 0,                             \\
	\end{align*}
	so it follows that
	\[S_{n-2}(\mathbf{x}) \ge \sum_{i \ne n-2, n} F_{n-2}(x_i) \ge F_{n-2}(x_1)+F_{n-2}(x_{n-1}) > 0,\]
	where of course the presence of the term $F_{n-2} (x_1)$ implicitly needs $n \ge 4$.

	We may extend the argument above to show for any $\lfloor n/2 \rfloor \le k \le n-2 $ that if $x_k<x_{k+1}=\dots=x_n$, then
	\[\mathcal{O}_{k}f_n(\mathbf{x}) =\frac{1}{n!}(x_k- x_n) S_k(\mathbf{x}) < 0.\]
	Indeed, we have $F_k(x_i)>0$ for each $1 \le i \le k-1$, and for each $k+1 \le j \le n-1$, we may pair up $F_k(x_j)$ and $F_k(x_{n-j})$ and check, as above, that $F_k(x_j) + F_k(x_{n-j}) > 0$. This then shows that
	\[
		S_{k}(\mathbf{x}) \ge \sum_{i \ne k, n} F_{n-2}(x_i) \ge \sum_{j = k+1}^{n-1}F_k(x_j) + F_k(x_{n-j}) > 0,\\
	\]
	as required.

	However, this pairing strategy fails once $\lfloor n/2 +1\rfloor$ or more of the largest coordinates of $\mathbf{x}$ are all equal. At this point, it will be convenient to treat the cases of even $n$ and odd $n$ separately.

	First, suppose that $n = 2\ell$ is even and that
	\[
		x_{k}<x_{k+1}=\dots=x_n\geq 1
	\]
	for some $1 \le k \le \ell - 1$.

	We reparametrise by setting $k=\ell - t -1$ for some $t\in\{0,1,\dots,\ell-2\}$, then set $y=\sqrt{x_{k}/x_n}$, and finally define
	\[
		f(y,t)=\sum_{i=1}^{n-2} c_{n-i} \binom{\ell + t}{i- 1}     - (n-2)! \left( \left(\ell + t\right) \frac{1}{y + 1} + \left(\ell - t -2\right) \frac{y}{y + 1} \right).
	\]

	Recall that
	\[
		S_{k} (\mathbf{x}) = \sum_{i=1}^{n-2} c_{n-i} e_{i-1}(\widehat{\mathbf{x}}_{k,n}) - (n-2)!\sum\limits_{i\neq k,n} \frac{\sqrt{x_i}}{\sqrt{x_{k}}+\sqrt{x_n}}.
	\]
	Since at least $\ell + t$ coordinates of $\mathbf{x}$ are bounded below by $1$, we have
	\[e_{i-1}(\widehat{\mathbf{x}}_{k,n}) \ge \binom{\ell + t}{i-1};\]
	it is also easy to see, since $x_1 \le \dots \le x_k$, that
	\[
		\sum_{1 \le i< k} \frac{\sqrt{x_i}}{\sqrt{x_{k}}+\sqrt{x_n}} \le \left(\ell - t -2\right) \frac{y}{y + 1},
	\]
	while the fact that $x_{k+1} = \dots = x_n$ gives us
	\[
		\sum\limits_{k < i \le n-1} \frac{\sqrt{x_i}}{\sqrt{x_{k}}+\sqrt{x_n}} = \frac{\ell + t}{y + 1}.
	\]
	It follows that $S_{k}(\mathbf{x}) \ge f(y,t)$, as claimed. The following lemma completes the argument in the case where $n$ is even.

	\begin{lemma}
		For all $y \in (0,1]$ and $t \in \{0,1,\dots,\ell-2\}$, we have $f(y,t)> 0$.
	\end{lemma}
	\begin{proof}
		Note that
		\[\frac{\partial f}{\partial y}=(n-2)!\frac{2t+2}{(1+y)^2}> 0,\]
		so $f$ is strictly increasing in $y$ for any fixed $t$. Thus, to prove the lemma, it suffices to show that $f(0,t) \geq  0$ for each $t \in \{0,1,\dots,\ell-2\}$; since $\partial f/ \partial y>0$ for all $y \in (0,1)$, the fact that $f(0,t)\geq 0$ implies that $f(y,t)>0$ for all $y \in (0,1]$.

		In what follows, we first show that $f(0,\ell-2)=0$, and then show that $f(0,t)$ increases as we decrease $t$, establishing the claim.

		Let us first check that
		\beq{kmaxcase}
		f(0, \ell-2) = \sum\limits_{i=1}^{n-2}c_{n-i}\binom{n-2}{i-1}-(n-2)!(n-2)=0.
		\enq
		Note that $(n-2)!(n-2)$ is the number of ways of permuting the set $[n-1]$ while ensuring that $n-1$ is not a fixed point. On the other hand, any such permutation has $i-1$ fixed points for some $1 \le i \le n-2$, so we may first select $i-1$ fixed points from $[n-2]$ in ${n-2\choose i-1}$ ways, and then select a derangement of the remaining $(n-1)-(i-1)=n-i$ elements in $c_{n-i}$ ways. The identity~\eqref{kmaxcase} is now immediate.

		Finally, we show that $f(0,t-1) > f(0,t)$ for all $1 \leq t \leq \ell-2$. Indeed, from the fact that $\ell + t - 1 < n-2$, we get
		\begin{align*}
			f(0,t)-f(0,t-1) & =\sum_{i=1}^{n-2} c_{n-i} \left(\binom{\ell + t}{i- 1} - \binom{\ell+t-1}{i-1}\right) -(n-2)! \\
			                & =\sum_{i=1}^{n-2} c_{n-i} \binom{\ell + t-1}{i - 2} - (n-2)!                                  \\
			                & = \sum_{j=0}^{n-4} c_{n-2-j} \binom{\ell + t-1}{j} - (n-2)!                                   \\
			                & < \sum_{j=0}^{n-4} c_{n-2-j} \binom{n-2}{j} - (n-2)!.
		\end{align*}
		Note that $(n-2)!$ is the number of permutations of $[n-2]$, and this may also be enumerated by counting permutations that fix exactly $j$ elements in $[n-2]$ for some $0\leq j\leq n-2$, yielding
		\[
			(n-2)! = \sum\limits_{j=0}^{n-2}c_{n-2-j}\binom{n-2}{j}.
		\]
		It follows that
		\begin{align*}
			f(0,t)-f(0,t-1) & < \sum_{j=0}^{n-4} c_{n-2-j} \binom{n-2}{j} - \sum\limits_{j=0}^{n-2}c_{n-2-j}\binom{n-2}{j} \\
			                & = - \sum\limits_{j=n-3}^{n-2}c_{n-2-j}{n-2\choose j} < 0,
		\end{align*}
		and the claim follows.
	\end{proof}

	Next, suppose that $n = 2\ell + 1$ is odd, and that
	\[
		x_{k}<x_{k+1}=\dots=x_n\geq 1
	\]
	for some $1 \le k \le \ell - 1$. The argument for this case is similar to the one in the even case above, but with some minor modifications.

	In this case, it is more convenient to reparametrise by setting $k=\ell - t$ for some $t\in\{1,\dots,\ell-1\}$, then set $y=\sqrt{x_{k}/x_n}$, and finally define
	\[
		g(y,t) = \sum_{i=1}^{n-2} c_{n-i} \binom{\ell + t}{i- 1} - (n-2)! \left( \left(\ell + t\right) \frac{1}{y + 1} + \left(\ell - t -1\right) \frac{y}{y + 1} \right).
	\]
	Arguing as in the previous case, it is not difficult to verify that
	\begin{enumerate}
		\item $S_{k}(\mathbf{x}) \ge g(y,t)$,
		\item $\partial g/ \partial y = (n-2)!({2t+1})/{(1+y)^2} > 0$ for all $y \in (0,1)$,
		\item $g(0,t-1)>g(0,t)$ for all $2 \leq t \leq \ell-1$, and
		\item $g(0, \ell-1) = 0$.
	\end{enumerate}
	These facts then imply that $S_{k}(\mathbf{x}) > 0$, as desired.

	We have shown, assuming $n \ge 4$, that the only critical point of $f_n$ in the interior of the domain $\mathfrak{C}$ is $\mathbf{1}$, that $f_n(\mathbf{1}) = 0$, and that $f_n (\mathbf{x}) > 0$ for all $\mathbf{x} \in \RR^n \setminus \mathfrak{C}$; this establishes the result for $n \ge 4$.

	To finish our proof of the result in full generality, we need to prove the claimed bound~\eqref{mainbd} for $n \le 3$.

	When $n=2$, the result is trivial: we need to verify that $1+x y \ge 2\sqrt{x y}$ for any $x, y \ge 0$, and this is immediate from the AM--GM inequality.

	When $n=3$, we have to show that
	\[
		2 + x + y + z + xyz \ge 2\sqrt{x y} + 2\sqrt{y z} + 2\sqrt{x z}
	\]
	for any $x, y, z \ge 0$. From the AM--GM inequality, we have
	\[2 + xyz = 1 + 1 + xyz \ge 3(xyz)^{1/3},\]
	so it suffices to show that
	\[
		x+y+z + 3(xyz)^{1/3} \ge 2\sqrt{x y} + 2\sqrt{y z} + 2\sqrt{x z}.
	\]
	Putting $x=a^3$, $y=b^3$ and $z=c^3$, this is equivalent to showing that
	\[
		a^3 + b^3 + c^3 + 3abc \ge 2\left((ab)^{3/2} + (bc)^{3/2} + (ac)^{3/2}\right)
	\]
	for $a,b,c \ge 0$. Schur's inequality tells us that
	\[
		a^3 + b^3 + c^3 + 3abc \geq ab(a+b) + bc(b+c) + ac(a+c),
	\]
	and the fact that
	\[
		ab(a+b) + bc(b+c) + ac(a+c) \ge 2\left((ab)^{3/2} + (bc)^{3/2} + (ac)^{3/2}\right)
	\]
	follows from three applications of the AM--GM inequality. Additionally, one may check that the only case of equality occurs when $x=y=z=1$. Hence, we have shown that the claim holds for $n=3$ as well, and the result now follows.
\end{proof}

\section{Conclusion}\label{sec:conc}
Our result raises the possibility that there exist other interesting families of \emph{inhomogeneous} inequalities satisfied by the elementary symmetric polynomials. While there is a fairly well-developed algebraic theory of inequalities for symmetric polynomials, this theory is mostly concerned with \emph{homogeneous} inequalities. While it is not clear how to extend the existing algebraic theory to prove our main result, we suspect that such an extension is likely to lead to some interesting new discoveries; we hope that our result will serve as a spur for further developments in this direction.

\section*{Acknowledgements}
The second author is grateful to the National Board for Higher Mathematics for a postdoctoral fellowship, and to the Indian Statistical Institute (Delhi Centre) for hosting him as a postdoctoral fellow. The third author was supported by NSF grant DMS-2237138 and a Sloan Research Fellowship.

\bibliographystyle{amsplain}
\bibliography{fixed_point_measures}

\end{document}